\documentclass[12pt]{amsart}

\usepackage{latexsym}
\usepackage{amsmath}
\usepackage{amsfonts}
\usepackage{amssymb}

\usepackage{mathtools}
\usepackage{verbatim}

\usepackage{array}
\newcolumntype{C}[1]{>{\centerring\arraybackslash}p{#1}}
\usepackage{booktabs}
\usepackage{enumitem}

\DeclarePairedDelimiter\ceil{\lceil}{\rceil}

\newtheorem{theorem}{Theorem}[section]
\newtheorem*{theorem*}{Theorem}

\newtheorem{corollary}[theorem]{Corollary}
\newtheorem{proposition}[theorem]{Proposition}

\theoremstyle{definition}

\newtheorem{example}[theorem]{Example}

\theoremstyle{remark}
\newtheorem{remark}[theorem]{Remark}

\numberwithin{equation}{section}

\begin{document}

\title{Congruences modulo powers of 11 for some eta-quotients.}

\author{Shashika Petta Mestrige}
\address{Mathematics Department\\
Louisiana State University\\
Baton Rouge, Louisiana}
\email{pchama1@lsu.edu}

\date{June 02, 2019}

\begin{abstract}
The partition function $ p_{[1^c11^d]}(n)$ can be defined using the generating function,
\[\sum_{n=0}^{\infty}p_{[1^c{11}^d]}(n)q^n=\prod_{n=1}^{\infty}\dfrac{1}{(1-q^n)^c(1-q^{11 n})^d}.\]
In this paper, we prove infinite families of congruences for the partition function $ p_{[1^c11^d]}(n)$ modulo powers of $11$ for any integers $c$ and $d$, which generalizes Atkin and Gordon's congruences for powers of the partition function. The proofs use an explicit basis for the vector space of modular functions of the congruence subgroup  $\Gamma_0(11)$.

\end{abstract}

\maketitle
\section{Introduction}
\label{introduction}

An (integer) partition of $n$ is a non-increasing sequence of positive integers $\lambda_1 \geq \lambda_2 \cdots \geq \lambda_r \geq 1$ that sum to $n$. Let $p(n)$ be the number of partitions of $n$. By convention, we take  $p(0)=1$ and $p(n)=0$ for negative $n$.

This function has been extensively studied in the last century. In the 1920's 
Ramanujan discovered amazing congruence properties for $p(n)$.

 \begin{theorem*}
    For all positive integers $j$, we have,
    \begin{align*}
p(5^jn+\delta_{5,j}) & \equiv0\pmod{5^j}, \\
p(7^jn+\delta_{7,j}) & \equiv0\pmod{7^{[\frac{j+2}{2}]}}, \\
p(11^jn+\delta_{11,j}) & \equiv0\pmod{11^j},
\end{align*}
    where  $24\delta_{\ell,j}\equiv1\pmod{\ell^j}$ for $\ell \in \{5,7,11\}$.
\end{theorem*}

Ramanujan in \cite{RS1}  proved the first two congruences for the case of $j=1$  by using the Jacobi triple product and later in \cite{RS2} using the theory of modular forms on $SL_2(\mathbb{Z})$. For arbitrary $j\geq 1$, Watson in \cite{W} gave a proof using modular equations for prime $5$ and $7$. Ramanujan in \cite{RS3} stated that he found a proof for the third congruence for $j=1,2$, but did not include the proof. In $1967$, Atkin in \cite{A2} gave a proof for the third congruence.

These fascinating congruence properties not only hold for the partition function itself, but also for the restricted partitions. To study a large class of restricted partitions, we study the partition function  $p_{[1^c\ell^d]}(n)$. This partition function also well studied in recent years, for example see Chan and Toh \cite{CHT}, and Liuquan Wang \cite{WL}. 

The partition function $p_{[1^c\ell^d]}(n)$ is defined using the generating function in the following way.

\[\prod_{n=1}^{\infty}\dfrac{1}{(1-q^n)^c(1-q^{\ell n})^d}=\sum_{n=0}^{\infty}p_{[1^c\ell^d]}(n)q^n.\]

To illustrate the importance of studying this partition function, let's look at the following four examples .

\begin{itemize}
    \item $d=0$ and $c>0$: $c-$color partitions.\\

    This partition function generates partitions of $n$ in to $c$ colors. See Gordon \cite{GB} and Atkin \cite{A1} for interesting congruence relations for primes less than or equal to 13.\\

    \item $c=1, d=-1$: $\ell$-regular partitions.\\

    This partition function generates partitions of $n$ with a restriction that no parts divisible by $\ell$. This partitions also well studied in recent years. See Wang \cite{WL2} and \cite{WL3} for divisibility properties of $5$- regular and $7$-regular partitions.\\

    \item $c=1, d=-\ell$: $\ell$-core partitions.\\

    This partition function generates partition of $n$ with a restriction that no hook numbers are divisible by $\ell$. See Wang \cite{WL} to see the divisibility properties of this partition function.\\

    \item $c=1, d=1$ and $\ell=2$: The cubic partition function.\\

    This partition function has a deep connection to the Ramanujan's cubic continued fraction, in \cite{CH} and \cite{CH2} Chan used this connection to obtain interesting congruences.
\end{itemize}

\vspace{3mm}

In 2016, in \cite{WL}, Wang proved the following congruences.\\
\begin{theorem*}[Wang, 2016]
For any integers $n\geq 0$ and $k\geq 1$,
\begin{align*}
p_{[1^111^{-11}]}\left(11^kn+11^k-5\right) & \equiv 0\pmod{11^k},\\
p_{[1^111^{-1}]}\left(11^{2k-1}n+\dfrac{7\cdot11^{2k-1}-5}{12}\right) & \equiv 0\pmod{11^k},\\
p_{[1^111^{1}]}\left(11^kn+\dfrac{11^k+1}{2}\right) & \equiv 0\pmod{ 11^k}.
\end{align*}
\end{theorem*}

Furthermore in \cite{WL}, he stated that it possible to obtain congruences for each value $c,d\in \mathbb{Z}$ separately.The primary goal of this paper is to find a unified way to prove congruences for the partition function $p_{[1^c{11}^d]}(n)$ for any $c,d\in \mathbf{Z}$.
\begin{theorem}\label{T:1.1}
For any integers $c$, $d$ and for any positive integer $r$,
\begin{equation}
    p_{[1^c{11}^d]}({11}^rm+n_r)\equiv 0\pmod{{11}^{A_r}}
\end{equation}
where $24n_r\equiv(c+{11}d)\pmod{{11}^r}$.
\end{theorem}
Here $A_r$  only depends on the integers $c,d$ and it can be calculated explicitly.\\

Moreover, we can obtain the following corollary, this is  similar to the Gordon's Theorem 1.2 in \cite{GB}.

\begin{corollary}\label{C:1.2}
For any positive integer $r$,
\[
 p_{[1^c11^d]}(11^rm+n_r)\equiv 0\pmod{11^{\frac{1}{2}\alpha r +\epsilon}}
 \]
where $24n_r\equiv(c+11d)\pmod{11^r}$, $\epsilon=\epsilon(c,d)=O(\log|c+11d|)$  and when $c+11d\geq 0$ , $\alpha$ depends on the residue of $c+11d\pmod{120}$ which is shown in Table \ref{T:1}.
 \begin{table}[htp]
\centering
\footnotesize\setlength{\tabcolsep}{2.3pt}
\begin{tabular}{l@{\hspace{6pt}} *{25}{c}}
\cmidrule(l){2-25}
& 1 & 2 & 3 & 4 & 5 & 6 & 7 & 8 & 9 & 10 & 11 & 12 & 13 & 14 & 15 & 16 & 17 & 18 & 19 & 20 & 21 & 22 & 23 & 24 \\
\midrule
\bfseries 0
 & 2 & 1 & 2 & 1 & 1 & 1 & 2 & 2 & 1 & 1 & 2 & 2 & 1 & 2 & 1 & 0 & 0 & 1 & 1 & 0 & 0 & 1 & 1 & 0\\
\bfseries 24
& 1 & 1 & 1 & 1 & 2 & 2 & 1 & 1 & 2 & 2 & 1 & 0 & 0 & 0 & 0 & 1 & 1 & 0 & 0 & 1 & 1 & 1 & 0 & 0\\
\bfseries 48
& 1 & 1 & 2 & 2 & 1 & 1 & 1 & 0 & 1 & 0 & 1 & 0 & 0 & 1 & 1 & 0 & 0 & 1 & 0 & 1 & 0 & 1 & 0 & 0\\
\bfseries 72
& 2 & 1 & 1 & 1 & 2 & 1 & 2 & 1 & 2 & 1 & 2 & 2 & 1 & 1 & 1 & 2 & 1 & 2 & 1 & 2 & 1 & 1 & 1 & 0\\
\bfseries 96
& 0 & 0 & 1 & 0 & 1 & 0 & 1 & 0 & 1 & 1 & 0 & 0 & 0 & 1 & 0 & 1 & 0 & 1 & 0 & 1 & 1 & 0 & 0 & 0\\
\bottomrule
\addlinespace
\end{tabular}
\caption{Values of $\alpha$}\label{T:1}
\end{table}

Here the entry is $\alpha(24i+j)$ where row labelled $24i$ and column labeled $j$. When $c+11d<0$, the last column must be changed to $2, 2, 2, 0, 2$.
\end{corollary}
\begin{remark}
This is the same shape as Gordon's result for $k-$colored partitions, with k replaced by $c+11d$.
\end{remark}

The remainder of the paper is as follows. Section \ref{prem} reviews properties of modular forms and operators on their coefficients. In Section \ref{proofs} the main results are proven. The paper closes with several examples in Section \ref{ex}.
 
\section{Preliminaries}\label{prem}
For a Laurent series $f(\tau)=\sum_{n\geq N} a(n)q^n$, we define the $U_p$ operator by,
\begin{equation}\label{E:2.1}
U_p\left(f(\tau)\right)=\sum_{pn\geq N} a(pn)q^n.
\end{equation}
\vspace{3mm}
Let $g(\tau)=\sum_{n\geq N} b(n)q^{n}$ be another Laurent series. The following simple property will play a key role in our proof.
\begin{equation}\label{E:2.2}
U_p\left(f(\tau)g(p\tau)\right)=g(\tau)U_p\left(f(\tau)\right).
\end{equation}

\begin{proposition}[\cite{AL}, lemma 7]\label{P:Gamma0N}
If $f(\tau)$ is a modular function for $\Gamma_0(N)$, if $p^2|N $, then $U_p\Big{(}f(\tau)\Big{)}$ is a modular function for $\Gamma_0(N/p)$.
\end{proposition}

\vspace{2mm}

Let $V$ be the vector space of modular functions on $\Gamma_0(11)$, which are holomorphic everywhere except possibly at 0 and $\infty$.\\

Atkin constructed a basis for $V$, Gordon in [2] slightly modified these basis elements and defined $J_{\nu}(\tau)$. For detailed information about the construction of the basis elements see \cite{A2}.

\begin{theorem}[Gordon \cite{GB}]
For all $\nu\in \mathbf{Z}$, we have:\\
\begin{enumerate}
    \item $J_\nu(\tau)=J_{\nu-5}(\tau)J_5(\tau)$,
    \item $\{J_\nu(\tau)|-\infty <\nu<\infty \}$ is a basis of $V$,
    \item ${\rm ord}_{\infty}J_\nu(\tau)=\nu$,
    \item $${\rm ord}_0J_\nu(\tau)=
\begin{cases}
		 -\nu \qquad & \text{if } \nu \equiv0\pmod{5}, \\
		 -\nu-1 & \text{if } \nu\equiv 1,2 \text{ or } 3\pmod{5}, \\
		-\nu-2 & \text{if } \nu\equiv 4\pmod{5}.
\end{cases}
$$
	\item The Fourier series of $J_\nu(\tau)$ has integer coefficients, and is of the form $J_\nu(\tau)=q^{\nu}+\dots$.
\end{enumerate}
\end{theorem}

Now let
\[\phi(\tau):=q^5\prod_{n=1}^{\infty}\left(\frac{1-q^{121n}}{1-q^n}\right)=\frac{\eta(121\tau)}{\eta(\tau)}.\]

This is a modular function on $\Gamma_0(121)$ by proposition \ref{P:Gamma0N}, hence $V$ is mapped to itself by the linear transformation,
\[T_{\lambda}:f(\tau)\rightarrow U_{11}\left(\phi(\tau)^{\lambda}f(\tau)\right)\]
where $\lambda$ is an integer. Let $(C_{\mu,v}^{\lambda})_{\mu,\nu}$ be the matrix of the linear transfomation $T_\lambda$ with respect to the basis elements $J_{\nu}$.

\begin{equation}\label{E:2.3}
U_{11}\left(\phi(\tau)^{\lambda}J_{\mu}(\tau)\right)
=\sum_{\nu}C_{\mu,\nu}^{\lambda}J_{\nu}(\tau)
\end{equation}

Gordon in \cite{GB}, proved an inequality (equation (17)) about the $11$-adic orders of the matrix elements (denoted by $\pi(C_{\mu,\nu})$).

\begin{equation}\label{E:2.4}
\pi(C_{\mu,\nu}^{\lambda})\geq \left[\dfrac{11{\nu}-\mu-5\lambda+\delta}{10}\right].
\end{equation}

Here $[x]$ means the floor function of the real number $x$ and $\delta=\delta(\mu,\nu)$ depends on the residues of $\mu$ and $\nu \pmod{5}$ according to the Table \ref{T:2}.\\
    \begin{center}
    \begin{table}[htp]
    \begin{tabular}{|l@{\hspace{16pt}}|*{6}{c|}}
\hline
& \multicolumn{5}{c|}{$\nu$} \\
\cline{1-6}
\hline
$\mu$& 0 & 1 & 2 & 3 & 4 \\
\hline\hline
\cline{1-6}
0 & -1 & 8 & 7 & 6 & 15\\
1 & 0 & 9 & 8 & 2 & 11 \\
2 & 1 & 10 & 4 & 13 & 12\\
3 & 2 & 6 & 5 & 4 & 13 \\
4 & 3 & 7 & 6 & 5 & 9\\ \hline
\end{tabular}\\
\vspace{2mm}
    \caption{Values of $\delta(\mu,\nu)$.}\label{T:2}
    \end{table}
    \end{center}
\vspace{-5mm}
We can clearly see from the table that $\delta \geq -1$, so we can rewrite \ref{E:2.4} as,

\begin{equation}\label{E:2.5}
\pi(C_{\mu,\nu}^{\lambda})\geq \left[\dfrac{11\nu-\mu-5\lambda-1}{10}\right].
\end{equation}

Now by Lemma $2.1(v)$, the Fourier series of $T_\lambda(J_{\mu})$ has all coefficients divisible by 11 if and only if,

$$C_{\mu,\nu}^{\lambda}\equiv0\pmod{11} \ \mbox{for all $\nu$}.$$

Now we define $\theta(\lambda,\mu)=1$ if all the coefficients of $U(\phi^{\lambda}J_{\mu})$ divisible by $11$. Otherwise we put $\theta(\lambda,\mu)=0$.\\

Now from the recurrences obtained in \cite{GB} page 119, we have

\begin{equation}\label{E:2.6}
\theta(\lambda-11,\mu)=
\theta(\lambda+12,\mu-5)=\theta(\lambda,\mu).
\end{equation}

Therefore the values of $\theta(\lambda,\mu)$ is completely determined by its values in the range $0\leq \lambda \leq 10$ and $0\leq \mu \leq 4$. Those values are listed in Table \ref{T:3}.

\begin{center}
\begin{table}[htp]
\begin{tabular}{|*{12}{c|}}
\hline
& \multicolumn{11}{c|}{$\lambda$} \\
\cline{1-12}
\hline
$\mu$ &0 & 1 & 2 & 3 & 4 & 5 & 6 & 7 & 8 & 9 & 10\\
\hline\hline
\cline{1-12}
0 & 0 & 1 & 0 & 1 & 0 & 1 & 0 & 1 & 1 & 0 & 0\\
1 & 1 & 1 & 0 & 1 & 0 & 0 & 0 & 1 & 1 & 0 & 0\\
2 & 1 & 1 & 1 & 0 & 0 & 0 & 0 & 1 & 1 & 0 & 0\\
3 & 1 & 0 & 1 & 0 & 0 & 0 & 0 & 1 & 1 & 0 & 0\\
4 & 1 & 0 & 1 & 0 & 1 & 0 & 1 & 1 & 0 & 0 & 0\\ \hline
\end{tabular}\\
\vspace{2mm}
\caption{Values of $\theta(\lambda, \mu)$.}\label{T:3}
\end{table}
\end{center}
Now we define,
\begin{equation}\label{E:2.7}
A_r(c,d):=\sum_{i=1}^{r-1}\theta(\lambda_{i},\mu_i),
\end{equation}
for any positive integer $r$ and integers $c,d$. We put $A_0:=0$.\\
\vspace{2mm}

We construct a sequence of modular functions that are the generating functions for the $p_{[1^c{11}^d]}(n)$ restricted to certain arithmetic progressions. This generalizes Gordon's construction for "$k-$color"  partitions. Here we use \eqref{E:2.2} repeatedly.

 $L_0 := 1$, and
\begin{align*}
L_1(\tau):&=U_{11}\left(\phi(\tau)^c\prod_{n=1}^{\infty}\frac{(1-q^{11n})^d}{(1-q^{11 n})^d}\right),\\
&=U_{11}\left(q^{5c}\prod_{n=1}^{\infty}\frac{(1-q^{121n})^c(1-q^{11n})^d}{(1-q^n)^c(1-q^{11 n})^d}\right),\\
&=\prod_{n=1}^{\infty}(1-q^{11 n})^c(1-q^n)^d\sum_{m\geq \lceil\frac{5c}{11}\rceil}^{\infty}p_{[1^c11^d]}(11m-5c)q^m.
\end{align*}

    Similarly we can define,
    \begin{align*}
        L_2(\tau) :&=U_{11}\left(\phi^d(\tau)L_1(\tau)\right),\\
    L_2(\tau)& =\prod_{n=1}^{\infty}(1-q^{11 n})^d(1-q^n)^c\sum_{m\geq \left\lceil\frac{5\cdot d+\lceil\frac{5c}{11}\rceil }{11}\right\rceil}^{\infty}p_{[1^c11^d]}(11^2m-5\cdot 11\cdot d -5\cdot c)q^m.
    \end{align*}

    Now, to get an equation for higher powers, We define,
    \begin{equation}\label{E:2.8}
    L_r:=U_{11}\left(\phi^{\lambda_{r-1}}(\tau)L_{r-1}\right),
    \end{equation}
    where
    \[ \quad \lambda_{r}=
\left\{
	\begin{array}{ll}
		 c & \mbox{if  $r$ is even }, \\
		d & \mbox{if  $r$ is odd}.
	\end{array}
\right. \]\\
	
	Then by a short calculation using \eqref{E:2.2} gives,
	
\begin{align}
\label{E:L_r}
    L_{2r}(\tau)& =\prod_{n=1}^{\infty}(1-q^n)^c(1-q^{11 n})^d\sum_{m\geq \mu_{2r}}p_{[1^c11^d]}(11^{2r}m+n_{2r})q^m, \\
    \notag
    L_{2r-1}(\tau)& =\prod_{n=1}^{\infty}(1-q^{11 n})^c(1-q^n)^d\sum_{m\geq \mu_{2r-1}}p_{[1^c11^d]}(11^{2r-1}m+n_{2r-1})q^m.
\end{align}
\vspace{3mm}

From \eqref{E:2.5}, \eqref{E:2.6} and \eqref{E:2.8} we can see that,
\[n_{2r}=-5\cdot d\cdot 11^{2r-1}+n_{2r-1},\]
\[n_{2r-1}=-5\cdot c\cdot 11^{2r-2}+n_{2r-2}.\]
Since $n_0=0$, using above recurrence relations we have,
\vspace{2mm}
\begin{itemize}[label={}]
    \item $n_1=-5\cdot c$
    \item $n_2=-5\cdot 11\cdot d -5\cdot c$
\end{itemize}
Using induction,
\begin{align}
\label{E:n_r}
n_{2r-1}&=-c\left(\dfrac{11^{2r}-1}{24}\right)-11\cdot d\left(\dfrac{11^{2r-2}-1}{24}\right).\\
\notag
n_{2r}&=-c\left(\dfrac{11^{2r}-1}{24}\right)-11\cdot d\left(\dfrac{11^{2r}-1}{24}\right) .
\end{align}
From this we have that,\\

    $24\cdot n_{2r-1}\equiv (c+11\cdot d) \mod 11^{2r-1}$  and
    $24\cdot n_{2r}\equiv (c+11\cdot d) \mod 11^{2r}$ .\\

    Therefore, for each $n_r$ are integers such that,
    \[24n_{r}\equiv (c+11\cdot d) \mod 11^r.\]

    Now, we need to find $\mu_r$ in terms of integers $c,d$. Notice that $\mu_r$ is the least integer $m$ such that $11^r m+n_r\geq 0$, which implies that,

    \begin{align}
    \label{E:mur}
    \mu_{2r-1}& =\left\lceil\frac{11\cdot c+d}{24}-\dfrac{c+11\cdot d}{24\cdot11^{2r-1}}\right\rceil,\\
    \notag
    \mu_{2r}& =\left\lceil\frac{c+11\cdot d}{24}-\dfrac{c+11\cdot d}{24\cdot11^{2r}}\right\rceil.
   \end{align}

Following Gordon, we represent these formulas in the following form.
\begin{align}
\label{E:2.11}
\mu_{2r-1}& =\left\lceil{\frac{11\cdot  c+d}{24}}\right\rceil+\omega(c,d) \quad  \mbox{if $|c+11\cdot d|<11^{2r-1}$}.\\
\notag
\mu_{2r}& =\left\lceil{\frac{c+11\cdot d}{24}}\right\rceil+\omega(c,d) \quad \mbox{if $|c+11\cdot d|< {11}^{2r}$ }.
\end{align}
\[\omega(c,d)=
\left\{
	\begin{array}{ll}
		 1 & \mbox{if  $c+ 11\cdot d<0$ and $24|(c+11\cdot d)$}, \\
		 0 & \mbox{Otherwise}. \\
	\end{array} \right. \]
\vspace{2mm}

\section{The proofs}\label{proofs}
\vspace{2mm}
If $$f(\tau):=\sum_{n\geq n_0}a(n)q^n,$$ we define,
\begin{equation}\label{E:3.1}
\pi\left(f(\tau)\right):=\mbox{min}_{n\geq n_0}\left\{\pi(a(n)\right\}.
\end{equation}

Here we follow Gordon's argument to prove $\pi(L_r)\geq A_r(c,d)$.

\begin{proof}[Proof of Theorem \ref{T:1.1}]
We see that using Proposition \ref{P:Gamma0N}, $L_r\in V$ for all $r$. So we have,
\begin{equation}\label{E:3.2}
L_r=\sum_{\nu}a_{r,\nu}J_{\nu}.
\end{equation}
Now by \eqref{E:2.1}, \eqref{E:2.8} and \eqref{E:3.2},
\begin{equation}\label{E:3.3}
    a_{r,\nu}=\sum_{\mu\geq\mu_{r-1}}a_{r-1,\mu}C_{\mu,\nu}^{\lambda_{r-1}}.
\end{equation}
\vspace{2mm}
Now we prove by induction,
\begin{equation}\label{E:3.4}
 \pi(L_r)\geq A_r+\left[\dfrac{\nu-\mu_r}{2}\right] \quad \mbox{for $\nu\geq \mu_r$}.
 \end{equation}

\begin{equation*}
\pi(a_{r,\nu}) \geq A_r + \left[\dfrac{\nu-\mu_r}{2}\right] \quad \text{for } \nu \geq \mu_r.
\end{equation*}

Since $A_0=0$, the result holds for $r=0$.
Now assume the result is true for $r-1$. Using \eqref{E:3.4} we have,
\begin{equation}\label{E:3.5}
\pi(a_{r-1,\nu})\geq A_{r-1}+\left[\dfrac{\nu-\mu_{r-1}}{2}\right] \quad \text{for $\nu\geq \mu_{r-1}$}.
\end{equation}

Now from equation \eqref{E:3.3} we have,
\begin{equation}\label{E:3.6}
\pi(a_{r,\nu})\geq \min_{\mu\geq\mu_{r-1}}\left(\pi(a_{r-1,\mu})+\pi(C_{\mu,\nu}^{\lambda_{r-1}})\right).
\end{equation}

From \eqref{E:2.5} and \eqref{E:3.5} the right hand side of \eqref{E:3.6} is at least equal to,
\begin{equation}
    A_{r-1}+\left[\dfrac{\mu-\mu_{r-1}}{2}\right]+\left[\dfrac{11\nu-\mu-5\lambda_{r-1}-1}{10}\right].
\end{equation}

This expression cannot decrease if $\mu$ is increased by $2$, so its minimum occurs when $\mu=\mu_{r-1}+1$, therefore at,
\begin{equation}\label{E:3.8}
A_{r-1}+\left[\dfrac{11\nu-\mu_{r-1}-5\lambda_{r-1}-2}{10}\right].
\end{equation}
Now from \eqref{E:2.8} we have,

 \begin{equation}\label{E:3.9}
     \mu_r=\left\lceil\dfrac{5\lambda_{r-1}+\mu_{r-1}}{11}\right\rceil,
 \end{equation}

therefore $\mu_r\geq\left(\dfrac{5\lambda_{r-1}+\mu_{r-1}}{11}\right).$\\

Plugging it in \eqref{E:3.8}, the right hand side of \eqref{E:3.5} is at least equal to
\begin{align*}
    A_{r-1}+\left[\dfrac{11\nu-11\mu_{r}-2}{10}\right] &=A_{r-1}+1+\left[\dfrac{11(\nu-\mu_r)-12}{10}\right]\\
    &\geq A_r+\left[\dfrac{\nu-\mu_r}{2}\right]  \mbox{ for all $\nu\geq\mu_r+2$,}
\end{align*}

since $A_{r-1}+1\geq A_r$.\\

Now consider $\nu=\mu_r$ or $\nu=\mu_r+1$.
\vspace{3mm}

If $\mu=\mu_{r-1},$
\[\pi(a_{r-1},\mu_{r-1})+\pi\left(C_{\mu_{r-1},\nu}^{\lambda_{r-1}}\right)\geq A_{r-1}+\theta(\mu_{r-1},\lambda_{r-1})=A_r.\]

This also works when $\mu\geq \mu_{r-1}+2$, since by induction hypothesis,

\[\pi(a_{r-1},\mu)\geq A_{r-1}+\left[\dfrac{\mu-\mu_{r-1}}{2}\right]\geq A_r.\]

Now consider $\mu=\mu_{r-1}+1$,
\vspace{3mm}

Now we need to show,
\[\pi(a_{r-1},\mu_{r-1}+1)+\pi(C_{\mu_{r-1}+1,\nu}^{\lambda_{r-1}})\geq A_{r-1}+\left[\dfrac{11\nu-(\mu_{r-1}+1)-5\lambda_{r-1}+\delta(\mu_{r-1}+1,\nu)}{10}\right].\]
\vspace{3mm}

Since $\nu=\mu_r$ or $\mu_r+1$, it suffices to show that when $\theta(\lambda_{r-1},\mu_{r-1})=1$,

\[\left[\dfrac{11\mu_r-\mu_{r-1}-1-5\lambda_{r-1}+\delta(\mu_{r-1}+1,\mu_r)}{10}\right]\geq 1,\]
and
\[ \left[\dfrac{11(\mu_r+1)-\mu_{r-1}-1-5\lambda_{r-1}+\delta(\mu_{r-1}+1,\mu_r+1)}{10}\right]\geq 1.\]
\vspace{3mm}

Now from \eqref{E:2.6}, Table \ref{T:2} and Table \ref{T:3} we see that the above claims hold.\\
\end{proof}

\begin{proof}[Proof of Corollary \ref{C:1.2}].\\

Recall \eqref{E:2.7},\\
\[A_r(c,d)=\sum_{i=0}^{r-1}\theta(\lambda_i,\mu_i).\]
Here we follow Gordon's argument from Section 4 of \cite{GB}, again with $k$ replaced by $c + 11d$. Note that Gordon's calculations had terms involving $11k$, which will be written in a more symmetric shape here using the fact that $11(c+11d) \equiv 11c + d \pmod{24}$.

\begin{align*}
A_r(c,d) & =\sum_{i=0}^{\log_{11}(c+11d)}\theta(\lambda_i,\mu_i)+\sum_{i=\log_{11}(c+11d)}^{r-1}\theta(\lambda_i,\mu_i) \\
&=\sum_{i=0}^{\log_{11}(c+11d)}\theta(\lambda_i,\mu_i)+N_1\cdot\theta\left(d,\ceil[\Big]{\dfrac{11c+d}{24}}+\omega(c,d)\right) \\
& \qquad \qquad \qquad \qquad +N_2\cdot\theta\left(c,\ceil[\Big]{\dfrac{c+11d}{24}}+\omega(c,d)\right).
\end{align*}

Here $N_1$ is the number of odd integers and $N_2$ is the number of even integers in the interval $\left[\log_{11}|c+11d|,r-1\right]$ respectively.\\

\begin{equation}
\mbox{Set} \quad \alpha:=\alpha(c,d)=\theta\left(d,\ceil[\Big]{\dfrac{11c+d}{24}}+\omega(c,d)\right)+\theta\left(c,\ceil[\Big]{\dfrac{c+11d}{24}}+\omega(c,d)\right)
.\end{equation}

Now if $r\leq \log_{11}|c+11d|+1$ then $N_1=N_2=0$,
\[A_r\leq \log_{11}|c+11d|.\]

If $r>\log_{11}|c+11d|+1$,
\[\left|N_1-\frac{1}{2}(r-1-\log_{11}(c+11d))\right|+\left|N_2-\frac{1}{2}(r-1-\log_{11}(c+11d))\right|<1.\]
\vspace{2mm}
Now consider,
\[\left|A_r-\frac{1}{2}\alpha(r-1-\log_{11}(c+11d))\right|<2+\log_{11}|c+11d|,\]
\[\left|A_r-\frac{\alpha r}{2} \right|<2+\frac{\alpha}{2}+(1+\frac{\alpha}{2})\log_{11}|c+11d|.\]\\
So we have $A_r=\frac{1}{2}\alpha r+\mathcal{O}\left(\log{}|c+11d|\right)$.
\vspace{2mm}\\

Now we  prove the condition for $\alpha$. As in [2], the proof is complete once we show that first $\alpha$ only depends $c+11d$ with the period $120$. Periodicity follows by the fact that $\alpha(c+11d)$ is invariant under the each maps,
\[c\rightarrow{c+120-11k} \quad\mbox{and}\quad d\rightarrow{d+k}\quad \mbox{for each integer k}.\]
\end{proof}.
\section{Examples}\label{ex}
Here we give three examples to demonstrate our method. The first two examples are Wang's results and the final one is a new example.\\

\begin{example}[$c=1$ and $d=1$]
\[A_r(1,1)=\sum_{i=0}^{r-1}\theta(\lambda_{i},\mu_{i}).\]
In this case, $\lambda_i=1$, $\mu_i=1$ for all $i$ and $\theta(1,1)=1$. So $A_r=r$ and $n_r=-\frac{11^r-1}{2}$.
thus we have,
\[p_{[1^1{11}^1]}\left(11^rm+\frac{11^r-1}{2}\right)\equiv 0\pmod{11^{r}}.\]
In [1], Wang obtained $n_r=-\frac{11^r+1}{2}$, which is slightly different from the $n_r$ given here, though one can obtain Wang's $n_r$
by setting $m \mapsto m-1$, thus both  congruences are equivalent.\\

Now, we illustrate corollary \ref{C:1.2} using this example. Since $c+11d=12$  from table \ref{T:1}, we have $\alpha=2$ so $A_r$ is asymptotically equals to $r$.
\end{example}

\begin{example}[$c=1$ and $d=-1$]
In this case, $\lambda_i$ is 1 if $i$ even or is -1 if $i$ is odd, we also have using \eqref{E:3.9}, $\mu_i$ is 1 if $i$ is odd and it is 0 if $i$ is even. By \eqref{E:n_r} and table \ref{T:3}, we have $n_{2r}=5\cdot \frac{11^{2r}-1}{12}$ and $A_{2r}=r$.
 \[p_{[1^111^{-1}]}\left(11^{2r}m+5\cdot \frac{11^{2r}-1}{12}\right)\equiv 0\pmod{11^{r}}.\]

In view of corollary \ref{C:1.2}, we have $\alpha=1$ from table \ref{T:1} since $c+11d=-10$. Notice here we used the fact that $\alpha$ is a periodic function of $c+11d$ with period 120. Now we have $A_{2r}$ asymptotically equals to $r$ with the error ‘$\epsilon’=0$.
\end{example}

\begin{example}[$c=2$ and $d=7$]

Then we have $\lambda_i$ is 2 if $i$ is even or 7 if  $i$ is odd, we also have $n_{2r}=\frac{7\cdot11^{2r}-79}{24}$ and $A_{2r}=2r-1$. In this case and the most cases $\mu_r$ are not immediately periodic. Here $\mu_0=0$ and $\mu_1=1$ afterwords, $\mu_i$ is $4$ if $i$ is odd or 2 if $i$ is even.

 \[p_{[1^211^{7}]}\left(11^{2r}m-\frac{7\cdot11^{2r}-79}{24}\right)\equiv 0\pmod{11^{2r-1}}.\]

 In this case $c+11d=79$, now from the table \ref{T:1}, $\alpha=2$ so $A_{2r}$ asymptotically $2r$ with the error $\epsilon=-1$.
 \end{example}

\section*{Acknowledgements}
The author thanks thesis advisor Karl Mahlburg for suggesting this problem and for his guidance during this project.

\end{document}